\documentclass[12pt]{amsart}
\usepackage{amsfonts}
\usepackage{amssymb,latexsym}
\usepackage{enumerate}
\usepackage{mathrsfs}
\makeatletter
\@namedef{subjclassname@2010}{%
  \textup{2010} Mathematics Subject Classification}
\makeatother

  [2002/01/19 v2.2g %
    AMS font definitions%
  ]
\DeclareFontFamily{U}{euf}{}
\DeclareFontShape{U}{euf}{m}{n}{%
  <5><6><7><8><9>gen*eufm%
  <10><10.95><12><14.4><17.28><20.74><24.88>eufm10%
  }{}
\DeclareFontShape{U}{euf}{b}{n}{%
  <5><6><7><8><9>gen*eufb%
  <10><10.95><12><14.4><17.28><20.74><24.88>eufb10%
  }{}

  [2002/01/19 v2.2g %
    AMS font definitions%
  ]
\DeclareFontFamily{U}{msb}{}
\DeclareFontShape{U}{msb}{m}{n}{%
  <5><6><7><8><9>gen*msbm%
  <10><10.95><12><14.4><17.28><20.74><24.88>msbm10%
  }{}

  [2002/01/19 v2.2g %
    AMS font definitions%
  ]
\DeclareFontFamily{U}{msa}{}
\DeclareFontShape{U}{msa}{m}{n}{%
  <5><6><7><8><9>gen*msam%
  <10><10.95><12><14.4><17.28><20.74><24.88>msam10%
  }{}

\newtheorem{theorem}{Theorem}[section]
\newtheorem{lemma}[theorem]{Lemma}

\theoremstyle{definition}
\newtheorem{definition}[theorem]{Definition}

\numberwithin{equation}{section}
\begin{document}

\title[On the exact degree of  multi-cyclic extension of $\mathbb{F}_{q}(t)$]
{On the exact degree of  multi-cyclic extension of $\mathbb{F}_{q}(t)$}

\author{Su Hu and Yan Li}

\address{Department of Mathematics, Korea Advanced Institute of Science and Technology (KAIST), 373-1 Guseong-dong, Yuseong-gu, Daejeon 305-701, South Korea
\\Present Address: Department of Mathematics and Statistics, McGill University, 805 Sherbrooke St. West, Montr\'eal, Qu\'ebec H3A 2K6, Canada}
\email{hus04@mails.tsinghua.edu.cn, hu@math.mcgill.ca}

\address{Department of Applied Mathematics, China Agriculture University, Beijing 100083, China}
\email{liyan\_00@mails.tsinghua.edu.cn}

\subjclass[2000]{11R58,11A15,11T23,11T24,11R45} \keywords{Power residue, Function
field, Kummer extension, Artin-Schreier extension, Character sum, Chebotarev's
Density Theorem.}

%\date{January 1, 2001 and, in revised form, June 22, 2001.}


\begin{abstract}
  Let
$q$ be a power of a prime number $p$, $k=\mathbb{F}_{q}(t)$ be the
rational function field over finite field $\mathbb{F}_{q}$ and $K/k$
be a multi-cyclic extension of prime degree. In this paper we will give an
exact formula for the degree of $K$ over $k$ by considering both Kummer and Artin-Schreier cases.
\end{abstract}

\maketitle

% we give a new approach to

\def\C{\mathbb C_p}
\def\BZ{\mathbb Z}
\def\Z{\mathbb Z_p}
\def\Q{\mathbb Q_p}
\def\C{\mathbb C_p}
\def\BZ{\mathbb Z}
\def\Z{\mathbb Z_p}
\def\Q{\mathbb Q_p}
\def\psum{\sideset{}{^{(p)}}\sum}
\def\pprod{\sideset{}{^{(p)}}\prod}

\section{Introduction}

Let $S=\{a_{1},...,a_{l}\}$ be a finite set of nonzero integers. By
computing the relative density of the set of prime numbers $p$ for
which all the $a_{i}'s$ are simultaneously  quadratic residues
modulo $p$, Balasubramanian, Luca and Thangadurai~\cite{Luca} gave
an exact formula for the degree of the multi-quadratic field
$\mathbb{Q}(\sqrt{a_{1}},...,\sqrt{a_{l}})$ over $\mathbb{Q}$.

Let
$q$ be a power of a prime number $p$, $k=\mathbb{F}_{q}(t)$ be the
rational function field over finite field $\mathbb{F}_{q}$ and $K/k$
be a multi-cyclic extension of prime degree. In this paper we will give an
exact formula of $K$ over $k$. We consider the following two different
situations. The first situation is  multi-cyclic
Kummer extensions. That is $K=k(\sqrt[m]{D_{1}},...,\sqrt[m]{D_{l}})$ and
$S=\{D_{1},...,D_{l}\}$ is a finite set of nonzero polynomials in $\mathbb{F}_{q}[t]$, where $m$ is a prime factor of $q-1$. The second situation is  multi-cyclic Artin-Schreier extensions. That is
$K=k(\alpha_{1},...,\alpha_{l})$ and there is a finite set
$S=\{D_{1},...,D_{l}\}$ of nonzero elements in
$\mathbb{F}_{q}(t)$ such that
$$\alpha_{i}^{p}-\alpha_{i}=D_{i} ~~(1\leq i\leq l).$$   We  follow  Balasubramanian, Luca and Thangadurai's  approach to consider the
above two situations in Section 2 and 3,  respectively. In these two sections, we also assume $K$ is a geometric
extension of $k$, i.e. the full constant field of $K$ is
$\mathbb{F}_{q}$ (see \cite[p.\,77]{Ro}). Our main tool is estimations of  certain character sums over $\mathbb F_q[T]$ (see~Lemma \ref{xiao1} and
 \ref{xiao2}, and the proof for Theorem~\ref{Kummer2}~and~\ref{AT2} below). In section 4, using abelian Kummer theory instead of Lemma \ref{xiao1} and
Lemma \ref{xiao2} we give another approach to this problem. Notice that in section 4 we do not assume $K/k$ is a geometric extension.

Throughout the paper, $\mathbb{C}$ denotes the complex field, $P$ denotes the monic irreducible polynomial, $N$ denotes a positive integer, $\pi(N)$ denotes the number of monic irreducible polynomial $P$ such that deg$P=N$ and $S_{k}$  denotes the set of monic irreducible polynomials which are unramified  in
$K$. A set $S$ of monic irreducible polynomials is said to have the relative density $\varepsilon$ with $0\leq \varepsilon\leq 1$,
if $$\varepsilon(S)=\lim\limits_{N\to\infty}\frac{\#\{P\in S~|~\textrm{deg}P=N\}}{\pi(N)}$$ exists.
In this case $$\varepsilon(S)=\lim\limits_{N\to\infty}\frac{\#\{P\in S~|~\textrm{deg}P\leq N\}}{\#\{P~|~\textrm{deg}P\leq N\}}$$ by Stolz's theorem.

A set $S$ of monic irreducible polynomials is said to have the Dirichlet density if
$$\delta(S)=\lim\limits_{s\to 1^{+}}\frac{\sum_{P\in S}NP^{-s}}{\sum_{P} NP^{-s}}$$ exists,
where $NP=q^{\textrm{deg}P}$, see \cite[p.\,126]{Ro}. Notice that the existence of  the Dirichlet density does not imply
the existence of  the  relative density, see Lemma 4.5 in~\cite{MS}.

The following lemmas will be used in the proof of our results.
\begin{lemma}~\label{Che}(Chebotarev's Density Theorem, first version, see Theorem 9.13A of ~\cite{Ro})
\\Let $K/k$  be a  Galois  extension of global function fields and set $H=\textrm{Gal}(K/k)$. Let $C\subset H$ be a conjugacy  class in $H$ and $S_{k}$  be the set of primes of  $k$ which are unramified  in
$K$. Then
$$\delta(\{P \in S_{k}|(P,K/k)=C\})=\frac{\#C}{\#H},
$$ where $\delta$  denotes the  Dirichlet  density and  $(P,K/k)$ is the Artin symbol at $P$.
\end{lemma}

\begin{lemma}~\label{Che2}(Chebotarev's Density Theorem, second version, see Theorem 9.13B of ~\cite{Ro})
\\Let $K/k$  be a  geometric, Galois  extension of global function fields and set $H=\textrm{Gal}(K/k)$. Let $C\subset H$ be a conjugacy  class in $H$. Suppose the common constant field of $K$ and $k$ has $q$ elements. Let $S_{k}$  be the set of primes of  $k$ which are unramified  in
$K$. Then for each positive integer $N$, we have
$$ \#\{P\in S_{k}~|~\textrm{deg}P=N, (P,K/k)=C\}=\frac{\#C}{\#H}\frac{q^{N}}{N}+ O\left(\frac{q^{N/2}}{N}\right).
$$
\end{lemma}

\begin{lemma}~\label{Che3}(The prime number theorem for polynomials, see Theorem 2.2 of ~\cite{Ro})
\\Let $N$  be a positive integer and $\pi(N)$ be the number of monic irreducible polynomial $P$  in $\mathbb{F}_{q}[t]$ of degree $N$. Then
$$ \pi(N)=\frac{q^{N}}{N}+ O\left(\frac{q^{N/2}}{N}\right).
$$
\end{lemma}

\section{Multi Kummer extensions}

In this section, let $q$ be a power of a prime and $m$ be any prime
divisor of $q-1$. Let $K$ be a multi-$m$-cyclic extension of
$k=\mathbb{F}_{q}(t)$. That is
$K=k(\sqrt[m]{D_{1}},...,\sqrt[m]{D_{l}})$ and
$S=\{D_{1},...,D_{l}\}$ is a finite set of nonconstant polynomials
in $A=\mathbb{F}_{q}[t]$. Let $\mathbb{Z}_m$ be the set of integers
$$\mathbb{Z}_m=\{0,1,2\cdots,m-2,m-1\} .$$ Let
$\gamma_S$ be the cardinality of the following set
$$\{(a_1,a_2\cdots,a_l)\in \mathbb{Z}_{m}^l~|~D_{1}^{a_{1}}D_{2}^{a_2}\cdots
D_{l}^{a_l}=F^m\ \mathrm{for\ some}\ F\in A \}.$$

In this section, we will prove the following result.

\begin{theorem}~\label{Kummer}
For a given finite set $S$ of nonzero polynomials with $|S|=l$, we
have $$[K:k]=m^{l-r},$$ where $r$ is the non-negative integer given
by $m^{r}=\gamma_{S}$.
\end{theorem}
Let $A=\mathbb{F}_{q}[t]$ and $A^{+}$ be the set of monic
polynomials in $A$. Let $P\in A$ be an irreducible polynomial and
$d$ be a divisor of $q-1$.

\begin{definition}(see ~\cite[p.\,24]{Ro})
If $P$ does not divide $a$, let $(a/P)_{d}$ be the unique elements
of $\mathbb{F}_{q}^{*}$ such that
$$a^{\frac{NP-1}{d}}\equiv\left(\frac{a}{P}\right)_{d} (\textrm{mod} P).$$ If $P|a$, define $(a/P)_{d}=0$. The symbol $(a/P)_{d}$ is called the $d$-th power residue symbol.
\end{definition}

If $m$ is a fixed prime divisor of $q-1$, then $\mathbb{F}_{q}^{*}$ has a unique subgroup $S_{m}$ of order $m$. Let $\eta$ be a fixed generator of $S_{m}$, i.e. $\eta\neq\bar{1}$ and $\eta^p=\bar{1}$, then $S_{m}=\{1,\eta,\eta^{2},...,\eta^{m-1}\}.$ We define $\chi$ to be the following
monomorphism:
$$\begin{aligned}\chi:\ \ S_{m}
&\rightarrow \mathbb{C}\\
 \eta^{k}&\mapsto\exp(\frac{2k\pi i}{m})\end{aligned}
$$
and we also denote
$$\left(\frac{a}{P}\right)=\chi\cdot\left(\frac{a}{P}\right)_{m},$$ for any $a\in A$ such that $P\nmid a$. If $P|a$, denote $$\left(\frac{a}{P}\right)=0.$$
We have
\begin{equation*}\label{Mul}\begin{aligned}\left(\frac{ab}{P}\right)=\left(\frac{a}{P}\right)\left(\frac{b}{P}\right),\end{aligned}\end{equation*}
for any $a,b\in A$, that is for any irreducible polynomial $P\in\mathbb{F}_{q}[t]$, $\left(\frac{\cdot}{P}\right)$ is a multiplicative character on $\mathbb{F}_{q}[t]$.

\begin{lemma}~\label{xiao1}
Let $E=k(\sqrt[m]{n})$ be a geometric Kummer extension of $k$. We
have
$$\sum_{\textrm{deg}P=N}\left(\left(\frac{n}{P}\right)+\left(\frac{n}{P}\right)^2+\cdots+\left(\frac{n}{P}\right)^{m-1}\right)=o(\pi(N)).$$
\end{lemma}
\begin{proof} Suppose Gal($E/k)=<\sigma>$. If $N$ is big enough, from  Proposition 10.6 in~\cite{Ro}, we have
\begin{equation*}
\begin{aligned}
T_{1}:&=\left\{\textrm{deg}P=N~|~\left(\frac{n}{P}\right)=1\right\}=\{P \in S_{k}~|~\textrm{deg}P=N, (P,E/k)=\rm{id}\},\\
T_{2}:&=\left\{\textrm{deg}P=N~|~\left(\frac{n}{P}\right)\neq
1\right\}=\{P \in S_{k}~|~\textrm{deg}P=N, (P,E/k)=\sigma^i, m\nmid
i\},
\end{aligned}
\end{equation*}
where $S_{k}$  is the set of monic irreducible polynomials which are
unramified  in $E$. By Chebotarev's Density Theorem (Lemma
\ref{Che2}), we have
\begin{equation*}
\begin{aligned}
\#T_{1}=&\frac{1}{m}\frac{q^{N}}{N}+
O\left(\frac{q^{N/2}}{N}\right);\\
\#T_{2}=&\frac{m-1}{m}\frac{q^{N}}{N}+
O\left(\frac{q^{N/2}}{N}\right).
\end{aligned}
\end{equation*}

Thus \begin{equation*}\begin{aligned} &\sum_{\substack{P\in T_{1}}
}\left(\left(\frac{n}{P}\right)+\left(\frac{n}{P}\right)^2+\cdots+\left(\frac{n}{P}\right)^{m-1}\right)\\=&\sum_{\substack{P\in
T_{1}}
}(1+1^2+\cdots+1^{m-1})=(m-1)\#T_{1}\\=&\frac{m-1}{m}\frac{q^{N}}{N}+
O\left(\frac{q^{N/2}}{N}\right);\\
&\sum_{\substack{P\in T_{2}}
}\left(\left(\frac{n}{P}\right)+\left(\frac{n}{P}\right)^2+\cdots+\left(\frac{n}{P}\right)^{m-1}\right)\\=&\sum_{\substack{P\in
T_{2}}
}(\zeta_{m}+\zeta_{m}^2+\cdots+\zeta_{m}^{m-1})=(-1)\#T_{2}\\=&-\frac{m-1}{m}\frac{q^{N}}{N}+
O\left(\frac{q^{N/2}}{N}\right),
\end{aligned}\end{equation*}
where $\zeta_{m}$ is some primitive $m$-th root of unity. Therefore,
\begin{equation*}\begin{aligned}
&\sum_{\textrm{deg}P=N}\left(\left(\frac{n}{P}\right)+\cdots+\left(\frac{n}{P}\right)^{m-1}\right)\\=&\sum_{P\in
T_{1}}\left(\left(\frac{n}{P}\right)+\cdots+\left(\frac{n}{P}\right)^{m-1}\right)+\sum_{P\in
T_{2}}\left(\left(\frac{n}{P}\right)+\cdots+\left(\frac{n}{P}\right)^{m-1}\right)\\=&O\left(\frac{q^{N/2}}{N}\right)
\end{aligned}\end{equation*}
 when $N$ is big
enough. Thus from the prime number theory for polynomials (Lemma
\ref{Che3}), we have
$$\sum_{\textrm{deg}P=N}\left(\left(\frac{n}{P}\right)+\left(\frac{n}{P}\right)^2+\cdots+\left(\frac{n}{P}\right)^{m-1}\right)=o(\pi(N)).$$

\end{proof}
\begin{lemma}
We have $\gamma_{S}=m^{r}$ for some $r\leq l$.
\end{lemma}
\begin{proof}
The proof is the same as Lemma 2.1 in~\cite{Luca} if we replace
$\mathbb{Q}^{*}/\mathbb{Q}^{*2}$ with $k^{*}/k^{*m}$.
\end{proof}

\begin{theorem}~\label{Kummer2}
Let $$\mathscr{M}:=\left\{P~|~\left(\frac{D_{1}}{P}\right)_{m}=...=\left(\frac{D_{l}}{P}\right)_{m}=1\right\}.$$
The relative density of $\mathscr{M}$ equals to
$$\frac{\gamma_{s}}{m^{l}}.$$
\end{theorem}
\begin{proof}
In fact  $$\begin{aligned}\mathscr{M}:&=\left\{P~|~\left(\frac{D_{1}}{P}\right)_{m}=...=\left(\frac{D_{l}}{P}\right)_{m}=1\right\}\\ &=\left\{P~|~\left(\frac{D_{1}}{P}\right)=...=\left(\frac{D_{l}}{P}\right)=1\right\}.\end{aligned}$$
Let $\mathscr{P}(S)$ be the set of all distinct prime factors of $D_{1}D_{2}...D_{l}$. Clearly, $\mathscr{P}(S)$ is a finite set.
Let $N$ be a positive integer. Considering the following counting
function:
$$R_{N}=\frac{1}{m^{l}}\sum_{\substack{\textrm{deg}P=N \\P\not\in\mathscr{P}(S)}}\left(1+\left(\frac{D_{1}}{P}\right)+\cdots+\left(\frac{D_{1}}{P}\right)^{m-1}\right)...\left(1+\left(\frac{D_{l}}{P}\right)+\cdots+\left(\frac{D_{l}}{P}\right)^{m-1}\right).$$
Since the $m$-th power residue symbol is completely multiplicative,
we have
\begin{equation*}\begin{aligned}
R_{N}&=\frac{1}{m^{l}}\sum_{\substack{\textrm{deg}P=N\\P\not\in\mathscr{P}(S)}}\sum_{\substack{
(b_{1},...,b_{l})\in
\mathbb{Z}_{m}^l\\n=D_{1}^{b_{1}}...D_{l}^{b_{l}}}}\left(\frac{n}{P}\right)
\\&=
\frac{1}{m^{l}}\sum_{\substack{\textrm{deg}P=N\\P\not\in\mathscr{P}(S)}}
\frac{1}{m-1}\biggl(\sum_{\substack{ (b_{1},...,b_{l})\in
\mathbb{Z}_{m}^l\\n=D_{1}^{b_{1}}...D_{l}^{b_{l}}}}
\left(\frac{n}{P}\right)+\sum_{\substack{ (b_{1},...,b_{l})\in
\mathbb{Z}_{m}^l\\n=D_{1}^{b_{1}}...D_{l}^{b_{l}}}}
\left(\frac{n^2}{P}\right)+\cdots\\&\quad+\sum_{\substack{ (b_{1},...,b_{l})\in
\mathbb{Z}_{m}^l\\n=D_{1}^{b_{1}}...D_{l}^{b_{l}}}}\left(\frac{n^{m-1}}{P}\right)\biggl)
\\&=
\frac{1}{m^{l}}\sum_{\substack{\textrm{deg}P=N\\P\not\in\mathscr{P}(S)}}\frac{1}{m-1}\sum_{\substack{
(b_{1},...,b_{l})\in
\mathbb{Z}_{m}^l\\n=D_{1}^{b_{1}}...D_{l}^{b_{l}}}}
\left(\left(\frac{n}{P}\right)+\left(\frac{n}{P}\right)^2+\cdots+\left(\frac{n}{P}\right)^{m-1}\right)
\\&=
\sum_{\substack{ (b_{1},...,b_{l})\in
\mathbb{Z}_{m}^l\\n=D_{1}^{b_{1}}...D_{l}^{b_{l}}}}\frac{1}{m^{l}}\sum_{\substack{\textrm{deg}P=N\\P\not\in\mathscr{P}(S)}}
\left(\left(\frac{n}{P}\right)+\left(\frac{n}{P}\right)^2+\cdots+\left(\frac{n}{P}\right)^{m-1}\right)\left(\frac{1}{m-1}\right).
\end{aligned}\end{equation*}
If $n$ is a perfect $m$-th power, then $\left(\frac{n}{P}\right)=1$
for each $P\not\in\mathscr{P}(S)$. Thus, for these $\gamma_{S}$
values of $n$, the inner sum is
$$\frac{1}{m^{l}}\sum_{\substack{\textrm{deg}P=N\\P\not\in\mathscr{P}(S)}}
\left(\left(\frac{n}{P}\right)+\cdots+\left(\frac{n}{P}\right)^{m-1}\right)
\left(\frac{1}{m-1}\right)=\frac{1}{m^{l}}\pi(N),$$ if $N$ is large
enough. For the remanning values of $n$ (i.e. when $n$ is not a
$m$-th power). From our assumption in this paper, $K/k$ is a
geometric extension. Then we have  $k(\sqrt[m]{n})/k$ is also a
geometric $m$-cyclic extension. Thus by Lemma ~\ref{xiao1}, we have
$$\frac{1}{m^{l}}\sum_{\substack{\textrm{deg}P=N\\P\not\in\mathscr{P}(S)}}
\left(\left(\frac{n}{P}\right)+\cdots+\left(\frac{n}{P}\right)^{m-1}\right)
\left(\frac{1}{m-1}\right)=o(\pi(N)),$$
 Therefore, if $N$ is large enough,
$$R_{N}=\frac{\gamma_{S}}{m^{l}}\pi(N)+o(\pi(N))$$
and hence
$$\frac{R_{N}}{\pi(N)}=\frac{\gamma_{S}}{m^{l}}+o(1).$$
So we have
$$\lim\limits_{N\to\infty}\frac{R_{N}}{\pi(N)}=\frac{\gamma_{S}}{m^{l}}.$$
This concludes the proof.
\end{proof}

Now we can prove the main result in this section.

Proof of Theorm~\ref{Kummer}: Let
$$f(x)=(x^{m}-D_{1})(x^{m}-D_{2})...(x^{m}-D_{l})\in A[x],$$ then
$K/k$ is the splitting field of $f(x)$. Let $S_{k}$  be the set of
monic irreducible polynomials which are unramified in $K$ and
$$\mathscr{M}:=\left\{P\in
S_{k}~|~\left(\frac{D_{1}}{P}\right)=...=\left(\frac{D_{l}}{P}\right)=1\right\}.$$
By Theorem~\ref{Kummer2}, we know that the relative density of $\mathscr{M}$
is
\begin{equation}~\label{K1}\frac{\gamma_{S}}{m^{l}}=\frac{1}{m^{l-r}}.
\end{equation}
Let $\sigma_{P}=(P,K/k)\in\textrm{Gal}(K/k)$. Since
$P\in\mathscr{M}$, $D_{i}$ is a $m$-th power residue modulo $P$ and
hence $P$ splits completely in $k(\sqrt{D_{i}})$. Therefore
$\sigma_{P}$ restricted to $k(\sqrt[m]{D_{i}})$ is the identity for
$1\leq i\leq l$. Suppose $[K:k]=m^{t}$. From our assumption, $K$ is
a geometric extension of $k$. By the Chebotarev Density Theorem
(Lemma~\ref{Che2}) and the prime number theorem for polynomials
(Lemma~\ref{Che3}), the relative density of $\mathscr{M}$ is
\begin{equation}~\label{K2}\frac{1}{[K:k]}=\frac{1}{m^{t}}.\end{equation} By comparing equations (\ref{K1}) and (\ref{K2}), we get $t=l-r$.

\section{Multi Artin-Schreier extensions}

In this section, let $q$ be a power of a prime number p and $K$ be a
multi-Artin-Schreier extension of $k$. That is
$K=k(\alpha_{1},...,\alpha_{l})$ and there is a finite set
$S=\{D_{1},...,D_{l}\}$ of nonconstant elements in
$k=\mathbb{F}_{q}(t)$ such that $$\alpha_{i}^{p}-\alpha_{i}=D_{i}
~~(1\leq i\leq l).$$ In the next subsection we recall the arithmetic
of Artin-Schreier extensions (also see~\cite{Hasse} and
~\cite{H-L}).

\subsection{The arithmetic of Artin-Schreier extensions}~\label{ATT}
Let $q$ be a power of a prime number $p$. Let $k=\mathbb{F}_{q}(t)$
be the rational function field. Let $L/k$ be a cyclic extension of
degree $p$. Then  $L/k$ is an Artin-Schreier extension, that is,
$L=k(\alpha)$, where $\alpha^{p}-\alpha=D,\ D\in\mathbb{F}_{q}(t)$
and that $D$ can not be written as $x^{p}-x$ for any $x\in k$.
Conversely, for any $D\in\mathbb{F}_{q}(t)$ and $D$ can not be
written as $x^{p}-x$ for any $x\in k$, $k(\alpha)/k$ is a cyclic
extension of degree $p$, where $\alpha^{p}-\alpha=D$. Two
Artin-Schreier extensions $k(\alpha)$ and $k(\beta)$ such that
$\alpha^p-\alpha=D$ and $\beta^p-\beta=D'$ are equal if and only if
they satisfy the following relations,
\begin{equation*}
  \begin{aligned}
          \alpha &\rightarrow x\alpha+B_{0}=\beta,\\
          D&\rightarrow xD+(B_{0}^{2}-B_{0})=D',\\
          x&\in\mathbb{F}_{p}^{*}, B_{0}\in k.
                           \end{aligned}
                           \end{equation*}
(See ~\cite{Hasse} or Artin~\cite{Artin} p.180-181 and p.203-206)
Thus we can normalize $D$ to satisfy the following conditions,

          $$D=\sum_{i=1}^{m}\frac{Q_{i}}{P_{i}^{e_{i}}} + f(t),$$
                  $$(P_{i},Q_{i})=1,\
\ \textrm{and}\ p\nmid e_i,\ \textrm{for}\ 1\leq i\leq m,$$
         $$p\nmid \textrm{deg}(f(t)),\ \textrm{if}\ f(t)\not\in \mathbb{F}_{q},$$
where $P_{i}\ (1\leq i\leq m)$ are monic irreducible polynomials in
$\mathbb{F}_{q}[t]$ and $Q_{i}\ (1\leq i\leq m)$ are polynomials in
$\mathbb{F}_{q}[t]$ such that $\textrm{deg} (Q_{i}) < \textrm{deg}
(P_i^{e_i})$.

If $D$ has
the above normalized forms, then the infinite
place $(1/t)$ is split, inert, or ramified in $L$ respectively when
$f(t)=0$; $f(t)$ is a constant and the equation $x^{p}-x=f(t)$ has
no solutions in $\mathbb{F}_{q}$; $f(t)$ is not a constant. Then the
field $K$ is called real, inert imaginary, or ramified imaginary
respectively. Moreover, the finite places of $k$ which are ramified in $K$ are
$P_{1},\cdots,P_{m}$ (see \cite[p.\,39]{Hasse}).
Let $P$ be
a finite place of $k$ which is unramified in $L$, i.e. $P$ does not equal to $P_{1},\cdots,P_{m}$. Let $(P,L/k)$ be
the Artin symbol at $P$. Then
\begin{equation}~\label{Artin}
(P,L/k)\alpha=\alpha+\{\frac{D}{P}\}
\end{equation} and the Hasse symbol
$\{\frac{D}{P}\}$
  is defined for $\rm{ord}_{P}(D)\geq 0$ by the following equalities:
\begin{equation}
~\label{HS}
\begin{aligned}
\{\frac{D}{P}\} &\equiv D+D^{p}+\cdots D^{N(P)/p}\ \textrm{mod} ~P\\ &\equiv (D+D^{q}+\cdots D^{N(P)/q}) \\
&+ (D+D^{q}+\cdots D^{N(P)/q})^{p}\\
&+ \cdots\\
&+  (D+D^{q}+\cdots D^{N(P)/q})^{q/p}\ \textrm{mod}
~P,\\\{\frac{D}{P}\} &=
\textrm{tr}_{\mathbb{F}_{q}/\mathbb{F}_{p}}\textrm{tr}_{(O_{K}/P)/\mathbb{F}_{q}}(D\
{\rm mod}\
  P)
\end{aligned}
\end{equation}
(see \cite[p.\,40]{Hasse}).
The
Artin-Schreier operator $\mathcal{P}$ is defined by
$$\mathcal{P}(x)=x^{p}-x,~\textrm{for}~ x\in L,$$ and obviously $\mathcal{P}$ is an additive operator.
A root of a polynomial $x^{p}-x-a$ with $a\in k$ will be denoted by $\mathcal{P}^{-1}(a)$ (see \cite[p.\,296]{Lang}).
Let $$\mathcal{P}k=\{\mathcal{P}(a) ~~|~~a\in k\}~~\textrm{and}~~\mathcal{P}^{-1}k=\{\mathcal{P}^{-1}(a) ~~|~~a\in k\},$$

We define $\varphi$ to be the following
monomorphism:
$$\begin{aligned}\varphi:\ \ \mathbb{F}_{p}
&\rightarrow \mathbb{C}\\
 x &\mapsto\exp(\frac{2x\pi i}{p})\end{aligned}
$$
and we also denote
$$\{\{\frac{D}{P}\}\}=\varphi\cdot\{\frac{D}{P}\},$$
for any $D\in K$ such that $\rm{ord}_{P}(D)\geq 0$.
We have
\begin{equation}\label{Mul}\begin{aligned}\{\{\frac{D_{1}+D_{2}}{P}\}\}=\{\{\frac{D_{1}}{P}\}\}
\{\{\frac{D_{2}}{P}\}\},\end{aligned}\end{equation}
for any $D_{1},D_{2}\in K$ such that $\rm{ord}_{P}(D_1)\geq 0$ and $\rm{ord}_{P}(D_1)\geq 0$.

\begin{lemma}~\label{xiao2}
Let $E=k(\mathcal{P}^{-1}(n)$) be a geometric Artin-Schreier extension of $k$.
We
have
$$\sum_{\textrm{deg}P=N}\left(\{\{\frac{n}{P}\}\}+\{\{\frac{n}{P}\}\}^2+\cdots+\{\{\frac{n}{P}\}\}^{p-1}\right)=o(\pi(N)).$$
\end{lemma}
\begin{proof} Suppose Gal($E/k)=<\sigma>$. If $N$ is big enough, from  equation~(\ref{Artin}), we have
\begin{equation*}
\begin{aligned}
T_{1}:&=\left\{\textrm{deg}P=N~|~\{\{\frac{n}{P}\}\}=1\right\}=\left\{\textrm{deg}P=N~|~\{\frac{n}{P}\}=0\right\}\\&=\{P \in S_{k}~|~\textrm{deg}P=N, (P,E/k)=\rm{id}\},\\
T_{2}:&=\left\{\textrm{deg}P=N~|~\{\{\frac{n}{P}\}\}\neq
1\right\}=\left\{\textrm{deg}P=N~|~\{\frac{n}{P}\}\neq
0\right\}\\&=\{P \in S_{k}~|~\textrm{deg}P=N, (P,E/k)=\sigma^i, p\nmid
i\},
\end{aligned}
\end{equation*}
where $S_{k}$  is the set of monic irreducible polynomials which are
unramified  in $E$. By Chebotarev's Density Theorem (Lemma
\ref{Che2}), we have
\begin{equation*}
\begin{aligned}
\#T_{1}=&\frac{1}{p}\frac{q^{N}}{N}+
O\left(\frac{q^{N/2}}{N}\right);\\
\#T_{2}=&\frac{p-1}{p}\frac{q^{N}}{N}+
O\left(\frac{q^{N/2}}{N}\right).
\end{aligned}
\end{equation*}

Thus \begin{equation*}\begin{aligned} &\sum_{\substack{P\in T_{1}}
}\left(\{\{\frac{n}{P}\}\}+\{\{\frac{n}{P}\}\}^2+\cdots+\{\{\frac{n}{P}\}\}^{p-1}\right)\\=&\sum_{\substack{P\in
T_{1}}
}(1+1^2+\cdots+1^{p-1})=(p-1)\#T_{1}\\=&\frac{p-1}{p}\frac{q^{N}}{N}+
O\left(\frac{q^{N/2}}{N}\right);\\
&\sum_{\substack{P\in T_{2}}
}\left(\{\{\frac{n}{P}\}\}+\{\{\frac{n}{P}\}\}^2+\cdots+\{\{\frac{n}{P}\}\}^{p-1}\right)\\=&\sum_{\substack{P\in
T_{2}}
}(\zeta_{p}+\zeta_{p}^2+\cdots+\zeta_{p}^{p-1})=(-1)\#T_{2}\\=&-\frac{p-1}{p}\frac{q^{N}}{N}+
O\left(\frac{q^{N/2}}{N}\right),
\end{aligned}\end{equation*}
where $\zeta_{p}$ is some primitive $p$-th root of unity. Therefore,
\begin{equation*}\begin{aligned}
&\sum_{\textrm{deg}P=N}\left(\{\{\frac{n}{P}\}\}+\cdots+\{\{\frac{n}{P}\}\}^{p-1}\right)\\=&\sum_{P\in
T_{1}}\left(\{\{\frac{n}{P}\}\}+\cdots+\{\{\frac{n}{P}\}\}^{p-1}\right)+\sum_{P\in
T_{2}}\left(\{\{\frac{n}{P}\}\}+\cdots+\{\{\frac{n}{P}\}\}^{p-1}\right)\\=&O\left(\frac{q^{N/2}}{N}\right)
\end{aligned}\end{equation*}
 when $N$ is big
enough. Thus from the prime number theory for polynomials (Lemma
\ref{Che3}), we have
$$\sum_{\textrm{deg}P=N}\left(\{\{\frac{n}{P}\}\}+\{\{\frac{n}{P}\}\}^2+\cdots+\{\{\frac{n}{P}\}\}^{p-1}\right)=o(\pi(N)).$$

\end{proof}

\subsection{Main result}

In this subsection we state and prove our main result.

Let $\gamma_S$ be the cardinality of the following set
$$\{(a_1,a_2\cdots,a_l)\in \mathbb{F}_{p}^l~|~a_{1}D_{1}+a_{2}D_{2}+\cdots
+a_{l}D_{l}=F^p-F\ \mathrm{for\ some}\ F\in k \}.$$

We will prove the following result.

\begin{theorem}~\label{AT1}
For a given finite set $S$ of nonconstant elements of $k=\mathbb{F}_q(t)$ with $|S|=l$, we have $$[K:k]=p^{l-r},$$ where $r$ is the non-negative integer given by $p^{r}=\gamma_S$.
\end{theorem}
\begin{lemma}
We have $\gamma_{S}=p^{r}$ for some $r\leq l$.
\end{lemma}
\begin{proof}
We extend the proof of Lemma 2.1 in~\cite{Luca} to our case. Let
$V=\mathbb{F}_{p}^{l}$ be the vector space having
$v_{1},...,v_{l}$ as a basis. Let $W=k/\mathcal{P}k$. Then $W$ is a
$\mathbb{F}_{p}$-vector space. Let $\tau:V\rightarrow W$ be
given by $\tau(v_{i})=D_{i}$ (mod $\mathcal{P}k$) and extended by
linearly. It is then clear that $(a_1,a_2\cdots,a_l)\in \mathbb{F}_{p}^l$ satisfies $a_{1}D_{1}+a_{2}D_{2}+\cdots
+a_{l}D_{l}\in\mathcal{P}k$ if and only if
$(a_1,a_2\cdots,a_l)\in\textrm{ker}(\tau)$. Thus
$\gamma_{S}=p^{r}$, where $r$ is the dimension of
ker($\tau$).
\end{proof}

\begin{theorem}~\label{AT2}
Let  $$\mathscr{M}:=\left\{P~|~\{\frac{D_{1}}{P}\}=...=\{\frac{D_{l}}{P}\}=0\right\}.$$ The relative density of $\mathscr{M}$ equals to
$$\frac{\gamma_{S}}{p^{l}}.$$
\end{theorem}
\begin{proof}
In fact $$\begin{aligned}\mathscr{M}:&=\left\{P~|~\{\frac{D_{1}}{P}\}=...=\{\frac{D_{l}}{P}\}=0\right\}\\&=\left\{P~|~\{\{\frac{D_{1}}{P}\}\}=...=\{\{\frac{D_{l}}{P}\}\}=1\right\}.\end{aligned}$$
 Let
$\mathscr{P}(S)=\cup_{i=1}^{l}S(D_{i})$, where $S(D_{i})$ is defined as the set of prime factors of the denominator of $D_i$. Clearly, $\mathscr{P}(S)$ is a finite set. Let $N$ be a positive integer. Considering the following counting
function:
$$R_{N}=\frac{1}{p^{l}}\sum_{\substack{\textrm{deg}P=N \\P\not\in\mathscr{P}(S)}}\left(1+\{\{\frac{D_{1}}{P}\}\}+\cdots+\{\{\frac{D_{1}}{P}\}\}^{p-1}\right)...\left(1+\{\{\frac{D_{l}}{P}\}\}+\cdots+\{\{\frac{D_{l}}{P}\}\}^{p-1}\right).$$
From equation~(\ref{Mul}), we have
\begin{equation*}\begin{aligned}
R_{N}&=\frac{1}{p^{l}}\sum_{\substack{\textrm{deg}P=N\\P\not\in\mathscr{P}(S)}}\sum_{\substack{
(b_{1},...,b_{l})\in
\mathbb{F}_{p}^l\\n=b_{1}D_{1}...+ b_{l}D_{l}}}\{\{\frac{n}{P}\}\}\\
&=\frac{1}{p^{l}}\sum_{\substack{\textrm{deg}P=N\\P\not\in\mathscr{P}(S)}}
\frac{1}{p-1}\biggl(\sum_{\substack{ (b_{1},...,b_{l})\in
\mathbb{F}_{p}^l\\n=b_{1}D_{1}+...+ b_{l}D_{l} }}
\{\{\frac{n}{P}\}\}+\sum_{\substack{ (b_{1},...,b_{l})\in
\mathbb{F}_{p}^l\\n=b_{1}D_{1}+...+ b_{l}D_{l} }}
\{\{\frac{2n}{P}\}\}+\cdots\\&\quad+\sum_{\substack{ (b_{1},...,b_{l})\in
\mathbb{F}_{p}^l\\n=b_{1}D_{1}+...+ b_{l}D_{l} }}\{\{\frac{(p-1)n}{P}\}\}\biggl)
\\&=
\frac{1}{p^{l}}\sum_{\substack{\textrm{deg}P=N\\P\not\in\mathscr{P}(S)}}\frac{1}{p-1}\sum_{\substack{
(b_{1},...,b_{l})\in
\mathbb{F}_{p}^l\\n=b_{1}D_{1}+...+ b_{l}D_{l} }}
\left(\{\{\frac{n}{P}\}\}+\cdots+\{\{\frac{n}{P}\}\}^{p-1}\right)
\\&=
\sum_{\substack{ (b_{1},...,b_{l})\in
\mathbb{F}_{p}^l\\n=b_{1}D_{1}+...+ b_{l}D_{l} }}\frac{1}{p^{l}}\sum_{\substack{\textrm{deg}P=N\\P\not\in\mathscr{P}(S)}}
\left(\{\{\frac{n}{P}\}\}+\cdots+\{\{\frac{n}{P}\}\}^{p-1}\right)\left(\frac{1}{p-1}\right).
\end{aligned}\end{equation*}
If $n\in\mathcal{P}k$, then $\{\{\frac{n}{P}\}\}=1$
for each $P\not\in\mathscr{P}(S)$. Thus, for these $\gamma_{S}$
values of $n$, the inner sum is
$$\frac{1}{p^{l}}\sum_{\substack{\textrm{deg}P=N\\P\not\in\mathscr{P}(S)}}
\left(\{\{\frac{n}{P}\}\}+\cdots+\{\{
\frac{n}{P}\}\}^{p-1}\right)
\left(\frac{1}{p-1}\right)=\frac{1}{p^{l}}\pi(N),$$ if $N$ is large
enough. For the remanning values of $n$ (i.e. $n\not\in\mathscr{P}k$). From
our assumption in this paper, $K/k$ is a geometric extension. we
have  $k(\mathcal{P}^{-1}(n))/k$ is also a geometric
extension. Thus by Lemma ~\ref{xiao2}, we have
$$\frac{1}{p^{l}}\sum_{\substack{\textrm{deg}P=N\\P\not\in\mathscr{P}(S)}}
\left(\{\{\frac{n}{P}\}\}+\cdots+\{\{\frac{n}{P}\}\}^{p-1}\right)
\left(\frac{1}{p-1}\right)=o(\pi(N)),$$
 Therefore, if $N$ is large enough,
$$R_{N}=\frac{\gamma_{S}}{p^{l}}\pi(N)+o(\pi(N))$$
and hence
$$\frac{R_{N}}{\pi(N)}=\frac{\gamma_{S}}{p^{l}}+o(1).$$
So we have
$$\lim\limits_{N\to\infty}\frac{R_{N}}{\pi(N)}=\frac{\gamma_{S}}{p^{l}}.$$
This concludes the proof.
\end{proof}

Now we can proof the main result in this section.

Proof of Theorm~\ref{AT1}:
Let  $$f(x)=(x^{p}-x-D_{1})(x^{p}-x-D_{2})...(x^{p}-x-D_{l})\in k[x],$$
then $K/k$ is the splitting field of $f(x)$. Let $S_{k}$  be the set of monic irreducible polynomials which are unramified  in
$K$
and  $$\mathscr{M}:=\left\{P\in S_{k}~|~\{\frac{D_{1}}{P}\}=...=\{\frac{D_{l}}{P}\}=0\right\}.$$
By Theorem~\ref{AT2}, we know that the relative density of $\mathscr{M}$ is
\begin{equation}~\label{A1}\frac{\gamma_{S}}{p^{l}}=\frac{1}{p^{l-r}}.\end{equation}
Let $\sigma_{P}=(P,K/k)\in\textrm{Gal}(K/k)$. Since $P\in \mathscr{M}$,
Hasse symbol
$\{\frac{D_{i}}{P}\}=0$, hence from equation~(\ref{Artin}),
$\sigma_{P}$ restricted to $k(\alpha_{i})$ is the identity, where
$$\alpha_{i}^{p}-\alpha_{i}=D_{i},$$ for $1\leq i\leq l$. Suppose
$[K:k]=p^{t}$. For our assumption, $K$ is a geometric extension of
$k$. By the Chebotarev's Density Theorem (Lemma~\ref{Che2}) and the
prime number theorem for polynomials (Lemma~\ref{Che3}), the
relative density of $\mathscr{M}$ is
\begin{equation}~\label{A2}\frac{1}{[K:k]}=\frac{1}{p^{t}}.\end{equation}
By comparing equations (\ref{A1}) and (\ref{A2}), we get $t=l-r$.
This finishes our proof.
\section{Another approach}
In this section, using abelian Kummer theory instead of Lemma \ref{xiao1} and
Lemma \ref{xiao2} we give another approach to this problem. Notice that in this section we do not assume $K/k$ is a geometric extension.

\subsection{Multi-Kummer case:} Let $m$ be any prime divisor of $q-1$. Let $K$ be a
multi-$m$-cyclic extension of $k=\mathbb{F}_{q}(t)$. That is
$K=k(\sqrt[m]{D_{1}},...,\sqrt[m]{D_{l}})$ and
$S=\{D_{1},...,D_{l}\}$ is a finite set of nonzero polynomials
in $A=\mathbb{F}_{q}[t]$.

Let $\mathbb{Z}_m$ be the set of integers
$$\mathbb{Z}_m=\{0,1,2\cdots,m-2,m-1\} .$$ Let
$\gamma_S$ be the cardinality of the following set
$$\{(a_1,a_2\cdots,a_l)\in \mathbb{Z}_m^l~|~D_{1}^{a_{1}}D_{2}^{a_2}\cdots
D_{l}^{a_l}=F^m\ \mathrm{for\ some}\ F\in\mathbb{F}_{q}[t] \},$$  We have

\begin{theorem}~\label{Kummer3}
$$[K:k]=m^{l-r},$$ where $r$ is the non-negative integer given by
$m^{r}=\gamma_{S}$.
\end{theorem}
\begin{proof}Let $B$ be a subgroup of $k^{*}$ generated by $k^{m}$
and $S$. From Chapter VI, Theorem 8.1 in~\cite{Lang} and the definition of $S$, we have
$$[K:k] = [k(B^{\frac{1}{m}}):k] = [B:k^{m}] = m^{l-r}.$$
\end{proof}

\begin{theorem}~\label{Kummer4}
Let $$\mathscr{M}:=\left\{P~|~\left(\frac{D_{1}}{P}\right)_{m}=...=\left(\frac{D_{l}}{P}\right)_{m}=1\right\}.$$ The Dirichlet density of $\mathscr{M}$ equals to
$1/m^{l-r},$ where $r$ is the non-negative integer given by
$m^{r}=\gamma_{S}$. In particular, if $k/k$ is a geometric extension, then the relative density of $\mathscr{M}$ also equals to
$1/m^{l-r}.$
\end{theorem}
\begin{proof}
Let $\sigma_{P}=(P,K/k)\in\textrm{Gal}(K/k)$. Since $P\in\mathscr{M}$,
$D_{i}$ is a $m$th power residue modulo $P$ and hence $P$ splits
completely in $k(\sqrt[m]{D_{i}})$ from  Proposition 10.6 in~\cite{Ro}. Therefore $\sigma_{P}$ restricted
to $k(\sqrt[m]{D_{i}})$ is the identity for $1\leq i\leq l$. From Theorem~\ref{Kummer3}, we have
$[K:k]=m^{l-r}$. By Chebotarev's Density Theorem (Lemma~\ref{Che}), the Dirichlet
density of $\mathscr{M}$ equals to $\frac{1}{[K:k]}=\frac{1}{m^{l-r}}.$ If $K/k$ is a geometric extension, then by Chebotarev's Density Theorem (Lemma~\ref{Che2}) and
the prime number theorem for polynomials (Lemma~\ref{Che3}), the relative
density of $\mathscr{M}$ also equals to $\frac{1}{[K:k]}=\frac{1}{m^{l-r}}.$
\end{proof}

\subsection{Multi-Artin-Schreier case:} Let $q$ be a power of prime $p$ and $K$ be a
multi-Artin-Schreier extension of $k=\mathbb{F}_{q}(t)$.
$K=k(\alpha_{1},...,\alpha_{l})$ and there is a finite set
$S=\{D_{1},...,D_{l}\}$ of nonzero elements in
$k$ such that $$\alpha_{i}^{p}-\alpha_{i}=D_{i}
~~(1\leq i\leq l).$$

Let $\gamma_S$ be the cardinality of the following set
$$\{(a_1,a_2\cdots,a_l)\in \mathbb{F}_{p}^l~|~a_{1}D_{1}+a_{2}D_{2}+\cdots
+a_{l}D_{l}=F^p-F\ \mathrm{for\ some}\ F\in k \}.$$ We have

\begin{theorem}~\label{AT3}
$$[K:k]=p^{l-r},$$ where $r$ is the non-negative integer given by
$p^{r}=\gamma_{S}$.
\end{theorem}
\begin{proof}Let $B$ be a subgroup of $k$ generated by $\mathcal{P}k$
and $S$. From Chapter VI, Theorem 8.3 in~\cite{Lang} and the definition of $S$, we have
$$[K:k] = [k(\mathcal{P}^{-1}B):k] = [B:\mathcal{P}k] = p^{l-r}.$$
\end{proof}
\begin{theorem}~\label{AT4}Let  $$\mathscr{M}:=\left\{P~|~\{\frac{D_{1}}{P}\}=...=\{\frac{D_{l}}{P}\}=0\right\}.$$ The Dirichlet density of $\mathscr{M}$ equals to
$1/p^{l-r},$ where $r$ is the non-negative integer given by
$p^{r}=\gamma_{S}$. In particular, if $k/k$ is a geometric extension, then the relative density of $\mathscr{M}$ also equals to
$1/p^{l-r}.$
\end{theorem}
\begin{proof}
The proof is similar to the proof of Theorem~\ref{AT1}.
Let $\sigma_{P}=(P,K/k)\in\textrm{Gal}(K/k)$. Since $P\in\mathscr{M}$,
If $\{\frac{D_{i}}{P}\}=1$, then $P$ splits
completely in $k(\alpha_{i})$ by equation~\ref{Artin}, where  $$\alpha_{i}^{p}-\alpha_{i}=D_{i}
~~(1\leq i\leq l).$$ Therefore $\sigma_{P}$ restricted
to $k(\alpha_{i})$ is the identity for $1\leq i\leq l$. From Theorem~\ref{AT3}, we have
$[K:k]=p^{l-r}$. By Chebotarev's Density Theorem (Lemma~\ref{Che}), the Dirichlet
density of $\mathscr{M}$ equals to $\frac{1}{[K:k]}=\frac{1}{p^{l-r}}.$ If $K/k$ is a geometric extension, then by Chebotarev's Density Theorem (Lemma~\ref{Che2}) and
the prime number theorem for polynomials (Lemma~\ref{Che3}), the relative
density of $\mathscr{M}$ also equals to $\frac{1}{[K:k]}=\frac{1}{m^{l-r}}.$
\end{proof}
\textbf{Acknowledgement:} This work is supported by National Natural
Science Foun- dation of China (Grant No. 11001145 and Grant No.
11071277).

\end{document}